\documentclass{amsart}

\numberwithin{equation}{section}

\usepackage[centertags]{amsmath}
\usepackage{amsfonts}
\usepackage{amssymb}
\usepackage{amsthm}
\usepackage{amscd}
\usepackage{hyperref}
\usepackage{algorithmic, algorithm}
\usepackage[all]{xy}

\usepackage{color}
\usepackage[dvipsnames]{xcolor}

\usepackage{mathrsfs} 

\DeclareFontFamily{OT1}{rsfs}{}
\DeclareFontShape{OT1}{rsfs}{n}{it}{<-> rsfs10}{}
\DeclareMathAlphabet{\mathscr}{OT1}{rsfs}{n}{it}

\newcommand{\cO}{\mathcal{O}}

\newcommand{\comment}[1]{}
\newcommand{\Q}{\mathbf{Q}}

\newcommand{\R}{\mathbf{R}}

\newcommand{\C}{\mathbf{C}}

\newcommand{\Z}{\mathbf{Z}}
\newcommand{\F}{\mathbf{F}}

\DeclareFontEncoding{OT2}{}{} 

\DeclareMathOperator{\ord}{ord}

\DeclareMathOperator{\Hom}{Hom}

\DeclareMathOperator{\Cl}{Cl}

\theoremstyle{plain} 
\newtheorem{thm}{Theorem}[section] 

\newtheorem{lem}[thm]{Lemma}

\theoremstyle{definition} 
\theoremstyle{remark} 
\newtheorem{rem}{Remark}

\newcounter{tasknumber}
\newcommand{\task}[2][]{%
  \addtocounter{tasknumber}{1}%
  \begin{center}%
  \framebox[1.1\width]{\begin{minipage}{0.9\textwidth}%
  \textbf{Task \arabic{tasknumber}} \textit{\if!#1(unassigned)!\else (#1)\fi}: {#2}%
  \end{minipage}}%
  \end{center}%
}

\newcounter{assumptionnumber}
\newcommand{\assumption}[2][]{%
  \addtocounter{assumptionnumber}{1}%
  \begin{center}%
  \framebox[1.1\width]{\begin{minipage}{0.9\textwidth}%
  \textbf{Assumption \arabic{assumptionnumber}} \textit{\if!#1!\else (#1)\fi}: {#2}%
  \end{minipage}}%
  \end{center}%
}

\newcommand{\authnote}[2][]{\noindent {\if!#1!  {\bf TODO} \else {\small \bf #1} \fi: #2}}

\newcommand{\real}{\mathfrak{Re}}

\begin{document}

\title[Generating subgroups of class groups]{Generating subgroups of ray class \\ groups with small prime ideals}

\author{Benjamin Wesolowski}
\address{\'Ecole Polytechnique F\'ed\'erale de Lausanne, EPFL IC LACAL, Switzerland}

\begin{abstract}
Explicit bounds are given on the norms of prime ideals generating arbitrary subgroups of ray class groups of number fields, assuming the Extended Riemann Hypothesis.
These are the first explicit bounds for this problem, and are significantly better than previously known asymptotic bounds.
Applied to the integers, they express that any subgroup of index $i$ of the multiplicative group of integers modulo $m$ is generated by prime numbers smaller than $16(i\log m)^2$, subject to the Riemann Hypothesis.
Two particular consequences relate to mathematical cryptology. Applied to cyclotomic fields, they provide explicit bounds on generators of the relative class group, needed in some previous work on the shortest vector problem on ideal lattices. Applied to Jacobians of hyperelliptic curves, they allow one to derive bounds on the degrees of isogenies required to make their horizontal isogeny graphs connected. Such isogeny graphs are used to study the discrete logarithm problem on said Jacobians.
\end{abstract}

\maketitle

\section{Introduction}

\subsection{Motivation}\label{subsec:motivation}
In 1990, Bach~\cite{Ba90} computed explicit bounds for the norms of prime ideals generating the class groups of number fields, assuming the Extended Riemann Hypothesis (henceforth, ERH). These bounds made explicit the earlier work of Lagarias, Montgomery and Odlyzko~\cite{LMO79}, and have proved to be a crucial tool in the design and analysis of many number theoretic algorithms.
However, these bounds do not tell anything about the norms of prime ideals generating any particular subgroup of the class group. Indeed, a generating set for the full group might not contain any element of the subgroup.

Let $K$ be a number field of degree $n$, and let $\Delta$ be the absolute value of its discriminant.
The results of~\cite{LMO79} show that the class group $\Cl(K)$ is generated by prime ideals of norm bounded by $O((\log \Delta)^{2})$.
Now, let $H$ be an arbitrary subgroup of the class group $\Cl(K)$.
Some asymptotic bounds on the norm of prime ideals generating $H$ have already been computed in~\cite{JW16} by analysing spectral properties of the underlying Cayley graphs. They are of the form $O((n[\Cl(K):H]\log \Delta)^{2+\varepsilon})$, for an arbitrary $\varepsilon > 0$.
Taking $H$ to be the full class group reveals a clear gap with the bounds of~\cite{LMO79}.
The explicit bounds provided in the present paper eliminate this gap, as they are asymptotically $O(([\Cl(K):H]\log \Delta)^{2})$.

Situations where proper subgroups of class groups have to be considered already arose in two distinct regions of mathematical cryptology. One is related to lattice-based cryptography.
Cryptographic schemes based on ideal lattices are typically instantiated over the ring of integers $\cO_K$ of a cyclotomic field $K$.
The field $K$ has a Hermitian vector space structure induced by its Minkowski embedding, and ideals of $\cO_K$ are also lattices in this vector space.
It was shown in~\cite{BS16,CGS14,CDPR16} that in principal ideals of $\cO_K$, an unusually short vector can be found in quantum polynomial time, under some heuristic assumptions (this short vector is actually a generator of the ideal). This led to the break of a multitude of cryptographic schemes using principal ideals (including~\cite{CGS14,GGH13,LSS14,SV10}).

A recent result~\cite{CDW16} shows how to extend the algorithm to find short vectors in arbitrary ideals of $\cO_K$, by transferring the problem to a principal ideal.
Let $n$ be the degree of $K$, $K_0$ the maximal real subfield of $K$, and $\Cl^-(K)$ the relative class group (i.e., the kernel of the norm map $\Cl(K) \rightarrow \Cl(K_0)$).
The transferring method of~\cite{CDW16} crucially relies on the assumption that $\Cl^-(K)$ is generated by a small number (polynomial in $\log n$) of prime ideals of small norm (polynomial in $n$) and all their Galois conjugates. On one hand, very little is known about the structure of $\Cl^-(K)$, and it seems difficult to prove that it can always be generated by such a small number of Galois orbits of ideals (yet there is convincing numerical evidence; see~\cite{Schoof98} for the case where $K$ has prime conductor). On the other hand it can be shown, assuming ERH, that the constraint on the norms can be satisfied, and the present work provides the best asymptotic bounds, and the first explicit ones (see Theorem~\ref{thmBoundRayClass} and Remark~\ref{rem:CDW}).

The second situation is related to hyperelliptic curves. 
Let $\mathscr A$ be the Jacobian of a hyperelliptic curve over a finite field $\F_q$. 
Isogeny graphs around $\mathscr A$ are a central tool to study the difficulty of the underlying discrete logarithm problem (see for instance~\cite{GHS02,JMV09,JW16,Smith09}).
When $\mathscr A$ is ordinary and absolutely simple --- as required for applications in cryptography --- its endomorphism algebra
is a complex multiplication field $K$ (with maximal real subfield $K_0$) and its endomorphism ring is isomorphic to an order $\cO$ in $K$.
Any abelian variety isogenous to $\mathscr A$ has the same endomorphism algebra, and an isogeny that also preserves the endomorphism ring is called a \emph{horizontal} isogeny.
The horizontal isogeny graphs of $\mathscr A$ are closely related to Cayley graphs of the kernel $\mathscr P(\cO)$ of the norm map
$$N_{K/K_0} : \Cl(\cO) \longrightarrow \Cl^+(\cO\cap K_0),$$ where $\Cl^+(\cO\cap K_0)$
is the narrow class group of $\cO\cap K_0$. More precisely, for any bound $B>0$, there is a graph isomorphism between
\begin{enumerate}
\item the Cayley graph of $\mathscr P(\cO)$ with generators the ideals of prime norm smaller than $B$, and
\item the isogeny graph consisting of all principally polarizable abelian varieties isogenous to $\mathscr A$ and with same endomorphism ring, and all isogenies between them of prime degree smaller than $B$.
\end{enumerate}
When the Jacobian $\mathscr A$ is an elliptic curve, the situation is well understood since $K_0 = \Q$, hence $\mathscr P(\cO) = \Cl(K)$. As a result, Bach's bounds have successfully been used to analyse various algorithms dealing with elliptic curve isogenies. In higher genus however, $\mathscr P(\cO)$ is typically a proper subgroup of the class group, and Bach's bounds are not sufficient to obtain connected isogeny graphs. New explicit bounds guaranteeing the connectedness are provided in Theorem~\ref{thmBoundIsogenies}.

\subsection{Setting.}
Throughout this paper, $K$ denotes a number field of degree $n$, with $r_1$ embeddings into $\R$ and $2r_2$ embeddings into $\C$.
Let $\mathscr I(K)$ denote the group of fractional ideals of the ring of integers $\cO_K$. A modulus $\mathfrak m$ of $K$ is a formal product of a finite part $\mathfrak m_0$ (an ideal in $\cO_K$), and an infinite part $\mathfrak m_\infty$ (a subset of the set of real embeddings of $K$). Then, $\mathscr I_\mathfrak m (K)$ denotes the subgroup generated by ideals coprime to $\mathfrak m_0$.

The notion of ray class group can now be recalled. Let $P_{K,1}^\mathfrak m$  be the subgroup of $\mathscr I_\mathfrak m(K)$ generated by principal ideals of the form $\alpha \cO_K$ where $\ord_\mathfrak p(\alpha - 1) \geq \ord_\mathfrak p(\mathfrak m_0)$ for all primes $\mathfrak p$ dividing $\mathfrak m_0$, and $\imath(\alpha) > 0$ for all $\imath \in \mathfrak m_\infty$.
The \emph{ray class group} of $K$ modulo $\mathfrak m$ is the quotient
$$\Cl_\mathfrak m(K) = \mathscr I_\mathfrak m(K)/ P_{K,1}^\mathfrak m.$$
For any ideal $\mathfrak a$ such that $(\mathfrak a, \mathfrak m) = 1$, let~$[\mathfrak a]_\mathfrak m $ denote its class in $\Cl_\mathfrak m(K)$.
The \emph{narrow class group} of $K$ is the group $\Cl_\mathfrak m(K)$ where $\mathfrak m$ is the set of all the real embeddings.

Our main tools to study these groups will be ray class characters. We call a \emph{ray class character modulo $\mathfrak m$} what Neukirch~\cite[Def.~VII.6.8]{Neukirch99} calls a (generalised) Dirichlet character modulo $\mathfrak m$, that is a Gr{\"o}{\ss}encharakter $\chi : \mathscr I_\mathfrak m(K) \rightarrow \C^\times$ that factors through the ray class group $\Cl_\mathfrak m(K)$ via the canonical projection.

\subsection{Main theorem.}
Let $K$ be a number field of degree $n$, and $\mathfrak m$ a modulus on $K$. Consider any subgroup $H$ of the ray class group $\Cl_\mathfrak m(K)$, and any character~$\chi$ that is not trivial on that subgroup. The main theorem generalizes~\cite{Ba90} by providing explicit bounds on the smallest prime ideal $\mathfrak p$ whose class is in $H$ and such that $\chi(\mathfrak p) \neq 1$. Note that all statements containing the mention (ERH) assume the Extended Riemann Hypothesis (recalled in Section~\ref{sec:rayclass}).
The following theorem is proved in Section~\ref{sec:mainproof}.

\begin{thm}[ERH]\label{mainTheorem}
Let $K$ be any number field, and $\Delta$ the absolute value of the discriminant of $K$.
Let $\mathfrak m$ be a modulus of $K$, with finite part $\mathfrak m_0$ and infinite part $\mathfrak m_\infty$.
Let $H$ be any subgroup of the ray class group $\Cl_\mathfrak m(K)$. Let $\chi$ be a ray class character modulo $\mathfrak m$ that is not trivial on $H$. Then there is a prime ideal $\mathfrak p$ such that $(\mathfrak p, \mathfrak m_0) = 1$, the class of $\mathfrak p$ in $\Cl_\mathfrak m(K)$ is in the subgroup $H$, $\chi(\mathfrak p) \neq 1$, $\deg( \mathfrak p) = 1$ and
$$N(\mathfrak p) \leq \left([\Cl_\mathfrak m(K):H]\left(2.71\log(\Delta N(\mathfrak m_0)) +  1.29|\mathfrak m_\infty| + 1.38\omega(\mathfrak m_0)\right) + 4.13\right)^2,$$
where $\omega(\mathfrak m_0)$ denotes the number of distinct prime ideals dividing $\mathfrak m_0$.
\end{thm}

\begin{rem}
When $H$ is the full group and $n \geq 2$, the above bound can be compared to Bach's bound $N(\mathfrak p) \leq 18 (\log(\Delta^2 N(\mathfrak m_0)))^2$ given by~\cite[Th. 4]{Ba90}. Let us put the expression of Theorem~\ref{mainTheorem} in a comparable form. From~\cite[Lem. 7.1]{Ba90}, we have 
$$|\mathfrak m_\infty| \leq n \leq \frac{\log(\Delta N(\mathfrak m_0)) + 3/2}{\log(2\pi) - \psi(2)} \leq 0.71 \log(\Delta N(\mathfrak m_0)) + 1.07,$$
where $\psi$ is the logarithmic derivative of the gamma function. Moreover, we have the bound $\omega(\mathfrak m_0) \leq \log(\Delta N(\mathfrak m_0))/\log 2$. The bound of Theorem~\ref{mainTheorem} becomes
$N(\mathfrak p) \leq \left(5.62\log(\Delta N(\mathfrak m_0)) +  5.52\right)^2.$
Whenever $\Delta N(\mathfrak m_0) < 12$, the corresponding ray class group is trivial, so we can suppose that $\log(\Delta N(\mathfrak m_0)) \geq \log(12) \geq 2.48$. These estimates lead to
\begin{equation}\label{eq:unrefinedForm}N(\mathfrak p) \leq \left(5.62 +  5.52/2.48\right)^2(\log(\Delta N(\mathfrak m_0)))^2 \leq 62(\log(\Delta N(\mathfrak m_0)))^2.
\end{equation}
Even in this form, direct comparison with~\cite[Lem. 7.1]{Ba90} is not obvious. With the unrefined estimate $\Delta^2 N(\mathfrak m_0) \leq (\Delta N(\mathfrak m_0))^2$, Bach's bound becomes $N(\mathfrak p) \leq 72 (\log(\Delta N(\mathfrak m_0)))^2$. The constant factor is slightly worse than in the bound~\eqref{eq:unrefinedForm}, but this comparison does not do justice to either theorem.
\end{rem}

\subsection{Consequences.}
In Section~\ref{sec:consequences}, a series of notable consequences is derived from Theorem~\ref{mainTheorem}.
Foremost, it allows us to obtain sets of small prime ideals generating any given subgroup of a ray class group. This is made precise in the following theorem.

\begin{thm}[ERH]\label{thmBoundRayClass}
Let $K$ be any number field, and $\Delta$ the absolute value of the discriminant of $K$.
Let $\mathfrak m$ be a modulus of $K$, with finite part $\mathfrak m_0$ and infinite part $\mathfrak m_\infty$.
Let $\mathfrak h$ be any ideal in $K$. Let $H$ be a non-trivial subgroup of the ray class group $\Cl_\mathfrak m(K)$. Then $H$ is generated by the classes of the prime ideals in
$$\{\mathfrak p \text{ prime ideal in } K  \mid (\mathfrak p, \mathfrak h\mathfrak m_0) = 1, [\mathfrak p]_\mathfrak m \in H, \deg(\mathfrak p) = 1 \text{ and } N(\mathfrak p) < B\},$$
where $B = \left([\Cl_\mathfrak m(K):H]\left(2.71\log(\Delta N(\mathfrak h\mathfrak m_0)) +  1.29|\mathfrak m_\infty| + 1.38\omega(\mathfrak h\mathfrak m_0)\right) + 4.13\right)^2$, and $[\mathfrak p]_\mathfrak m$ denotes the class of $\mathfrak p$ in $\Cl_\mathfrak m(K)$.
\end{thm}

\begin{rem}\label{rem:CDW}
In particular, Theorem~\ref{thmBoundRayClass} implies that the relative class group of a cyclotomic field $K$ of degree $n$ and discriminant $\Delta$ is generated by ideals of prime norm smaller than $\left(2.71 h_{K_0}\log \Delta + 4.13\right)^2$, where $h_{K_0}$ is the class number of the maximal real subfield of~$K$. This is an important improvement for~\cite{CDW16} over the previously known bound $O((h_{K_0}n \log \Delta )^{2+\varepsilon})$ derived from~\cite{JW16}. 
\end{rem}

Applying Theorem~\ref{mainTheorem} to Dirichlet characters, one can obtain new results on subgroups of the multiplicative group $(\Z/m\Z)^\times$.
Let $m$ be a positive integer, and $H$ a non-trivial subgroup of $G = (\Z/m\Z)^\times$.
It is already known that, assuming GRH, $H$ contains a prime number smaller than $O(([G:H]\log m)^2)$ (see~\cite{BS96,LLS15}). But these bounds do not provide a generating set for $H$: they only guarantee the existence of one such prime number.
The following theorem gives a set of generators of $H$, whose norms are also asymptotically $O(([G:H]\log m)^2)$.

\begin{thm}[ERH]\label{thmBoundIntegers}
Let $m$ be a positive integer, and $H$ a non-trivial subgroup of $G = (\Z/m\Z)^\times$. Then $H$ is generated by the set of prime numbers $p$ such that $p \mathrm{\ mod\ } m \in H$ and $p \leq 16\left([G:H]\log m\right)^2$. 
\end{thm}

Finally, we derive bounds on the degrees of cyclic isogenies required to connect all isogenous principally polarizable abelian varieties over a finite field sharing the same endomorphism ring.

\begin{thm}[ERH]\label{thmBoundIsogenies}
Let $\mathscr A$ be a principally polarized, absolutely simple, ordinary abelian variety over a finite field $\F_q$, with endomorphism algebra $K$ and endomorphism ring isomorphic to an order $\cO$ in $K$. Let $K_0$ be the maximal real subfield of $K$, and $\mathfrak f$ the conductor of $\cO$. For any $B > 0$, let $\mathscr G(B)$ be the isogeny graph whose vertices are the principally polarizable varieties isogenous to $\mathscr A$ and with the same endomorphism ring, and whose edges are isogenies connecting them, of prime degree smaller than $B$.
Then, if $\cO_0 = \cO \cap K_0$ is the ring of integers of $K_0$, the  graph $$\mathscr G\left(26\left(h_{\cO_0}^+\log(\Delta N(\mathfrak f)) \right)^2\right)$$ is connected, with $\Delta$ the absolute value of the discriminant of $K$, and $h_{\cO_0}^+$ the narrow class number of $\cO_0$.
\end{thm}

\begin{rem}
In particular, the above holds in dimension 2, where \emph{principally polarized} translates to \emph{Jacobian of a genus 2 hyperelliptic  curve} (see~\cite[Th. 4.1]{MN02}).
\end{rem}

\subsection{Notation.}
An inequality such as $x \leq y$ between complex numbers means that the relation holds between the real parts. 
The function $\log$ denotes the natural logarithm.

\section{Ray class characters}\label{sec:rayclass}

This section summarizes the definitions, notations and facts related to ray class characters that will be used throughout the paper.

Recall that a ray class character modulo $\mathfrak m$ is a Gr{\"o}{\ss}encharakter $\chi : \mathscr I_\mathfrak m(K) \rightarrow \C^\times$ that factors through the ray class group $\Cl_\mathfrak m(K)$ (via the canonical projection).
A character is \emph{principal} if it takes only the value 1. Let $\delta(\chi)$ be $1$ if $\chi$ is principal and $0$ otherwise.
A ray class character is \emph{primitive modulo $\mathfrak m$} if it does not factor through $\Cl_{\mathfrak m'}(K)$ for any modulus $\mathfrak m'$ smaller\footnote{A modulus $\mathfrak m'$ is (strictly) smaller than $\mathfrak m$ if $\mathfrak m_0' \mid \mathfrak m_0$, $\mathfrak m'_\infty \subseteq \mathfrak m_\infty$ and $\mathfrak m' \neq \mathfrak m$.} than~$\mathfrak m$. The conductor $\mathfrak f_\chi$ of $\chi$ is the smallest modulus $\mathfrak f$ such that $\chi$ is the restriction of a ray class character modulo~$\mathfrak f$. Let $\beta_\chi = |\mathfrak f_\infty|$ be the number of infinite places in the conductor $\mathfrak f$. From~\cite[Prop.~6.9]{Neukirch99}, any ray class character $\chi$ is the restriction of a primitive ray class character of modulus~$\mathfrak f_\chi$, which is also primitive as a Gr{\"o}{\ss}encharakter. 

The Hecke $L$-function associated to a character $\chi$ modulo $\mathfrak m$ is defined as
$$L_\chi(s) = \sum_{\mathfrak a}\frac{\chi(\mathfrak a)}{N(\mathfrak a)^s},$$
for $\real(s)>1$, where the sum is taken over all ideals of $\cO_K$. Note that $\chi$ is implicitly extended to all ideals by defining $\chi(\mathfrak a) = 0$ whenever $(\mathfrak a, \mathfrak m_0) \neq 1$.
When $\chi$ is the trivial character on $\mathscr I(K)$, we obtain the Dedekind zeta function of $K$,
$\zeta_K(s) = \sum_{\mathfrak a}N(\mathfrak a)^{-s}.$
These $L$-functions are extended meromorphically on the complex plane with at most a simple pole at $s = 1$, which occurs if and only if $\chi$ is principal.
Let $R_\chi$ be the set of zeros of $L_\chi$ on the critical strip $0 < \real(s) < 1$.
The ERH implies that all Hecke $L$-functions are zero-free in the half-plane $\real(s)>1/2$.

We will make an extensive use of the logarithmic derivatives $L'_\chi/L_\chi$. When ${\real(s)>1}$, they admit the absolutely convergent representation
\begin{equation}\label{eq:formulaLogDerivative}\frac{L'_\chi}{L_\chi}(s) = -\sum_\mathfrak a \frac{\Lambda(\mathfrak a)\chi(\mathfrak a)}{N(\mathfrak a)^s},\end{equation}
where $\Lambda$ is the von Mangoldt function (i.e., $\Lambda(\mathfrak a) = \log N(\mathfrak p)$ if $\mathfrak a$ is a power of a prime ideal $\mathfrak p$, and $0$ otherwise).
The residues of ${L'_\chi}/{L_\chi}$ when $\chi$ is primitive modulo $\mathfrak m$ are summarised in Table~\ref{tab:residues}, which comes from~\cite[p. 361]{Ba90} (with the observation that~$\beta$ in~\cite{Ba90} coincides with~$\beta_\chi = |\mathfrak m_\infty|$ for characters $\chi$ which are primitive modulo $\mathfrak m$).

Let $\psi$ be the logarithmic derivative of the gamma function, and for any ray class character $\chi$ on $K$, define
\begin{equation}\label{eq:logdgamma}
\psi_\chi(s) = \frac{r_1+r_2 - \beta_\chi}{2}\psi\left(\frac{s}{2}\right) + \frac{r_2 + \beta_\chi}{2}\psi\left(\frac{s+1}{2}\right) - \frac{n\log \pi}{2}.
\end{equation}
The main reason to introduce these functions is the following formula: for any complex number $s$, if $\chi$ is primitive then
\begin{equation}\label{eq:logdlchi}
 -\real\frac{L'_\chi}{L_\chi}(s) =  \frac{1}{2}\log(\Delta N(\mathfrak f_{\chi})) + \real \left(\delta(\chi)\left(\frac{1}{s}+\frac{1}{s-1}\right) -\sum_{\rho \in R_{\chi}}\frac{1}{s - \rho} + \psi_\chi(s)\right).
 \end{equation}
A proof can be found in~\cite[Lem. 5.1]{LO77}.

\begin{table}[!t]
\caption{Residues of the logarithmic derivative of Hecke $L$-functions, when $\chi$ is a primitive ray class character~(\cite[p. 361]{Ba90}).}
\label{tab:residues}
\centering
\begin{tabular}{c | c | c}
 \textbf{place} & \textbf{residue of $\zeta_K'/\zeta_K$}  & \textbf{residue of  ${L'_\chi}/{L_\chi}$}\\
1  &  $-1$ & 0 \\
$\rho \in R_1$ & 1 & \text{$0$ if $\rho \not\in R_\chi$, $1$ otherwise}\\
$\rho \in R_\chi$ & \text{$0$ if $\rho \not\in R_1$, $1$ otherwise} & 1\\
0 & $r_1 + r_2 -1$ & $r_1 + r_2 - \beta_\chi$\\
$-2n + 1, n \in \mathbf N_{>0}$ & $r_2$ & $r_2 + \beta_\chi$\\
$-2n, n \in \mathbf N_{>0}$ & $r_1 + r_2$ & $r_1 + r_2 - \beta_\chi$\\
\end{tabular}
\end{table}

\section{Proof of the main theorem}\label{sec:mainproof}
Throughout this section, consider a ray class character $\chi$ modulo $\mathfrak m$ that is not trivial on a given subgroup $H$ of $G = \Cl_\mathfrak m(K)$. 

\subsection{Outline of the proof.}
For any $0<a<1$, $x>0$, and ideal $\mathfrak a$, let
$$P(\mathfrak a,x) = \Lambda(\mathfrak a)\left(\frac{N(\mathfrak a)}{x}\right)^a \log \left(\frac{x}{N(\mathfrak a)}\right).$$
Let us start by recalling a lemma that is the starting point of the original proof of Bach's bounds.

\begin{lem}[\protect{\cite[Lem. 4.2]{Ba90}}]\label{lemmaBachInt}
For $0 < a < 1$ and any character $\eta$, 

\[\sum_{\substack{N(\mathfrak a) < x}} \eta(\mathfrak a)P(\mathfrak a,x)
=  \frac{-1}{2\pi i} \int_{2-i\infty}^{2+i\infty}\frac{x^s}{(s+a)^2}\cdot \frac{L'_{\eta}}{L_{\eta}}(s)ds.\]
\end{lem}

Bach then considers the difference between two instances of this equality at $\eta = 1$ and at $\eta = \chi$, and proves the bounds by estimating the right-hand side as $x + O(\sqrt{x})$, while the left-hand side is zero if the character is trivial on all prime ideals of norm smaller than $x$; therefore such an $x$ cannot be too large.

The proof of Theorem~\ref{mainTheorem} follows the same strategy. It exploits the series of lemmata provided in~\cite[Sec. 5]{Ba90}, interlacing them with a game of characters of $G/H$ in order to account for the new condition $[\mathfrak a]_\mathfrak m  \in H$.
Consider the group of characters  of the quotient $G/H$, namely $\widehat{G/H} = \Hom(G/H, \C^\times)$. Given any character $\theta \in \widehat{G/H}$, let $\theta^*$ be the primitive ray class character such that $\theta^*(\mathfrak a) = \theta([\mathfrak a]_\mathfrak m H)$ whenever $(\mathfrak a, \mathfrak m_0) = 1$.
For any $\theta \in \widehat{G/H}$, write $L_\theta$ for the $L$-function of $\theta^*$.
For any ray class character $\eta$ and any $\theta \in \widehat{G/H}$, let $\eta_\theta$ denote the primitive character inducing the product $\eta\theta^*$. 

\begin{lem}\label{lemma:sumTheta}
Let $\mathfrak a$ be any ideal in $K$. Let $\mathfrak n_0$ be the largest divisor of $\mathfrak m_0$ coprime to $\mathfrak a$, and $\mathfrak n = \mathfrak n_0\mathfrak m_\infty$. Let $\pi : \Cl_{\mathfrak m}(K) \rightarrow \Cl_{\mathfrak n}(K)$ be the natural projection. Then,
$$
\sum_{\theta \in \widehat{G/H}} \theta^*(\mathfrak a) = 
\begin{cases}
[\Cl_{\mathfrak n}(K):\pi(H)] & \text{if } [\mathfrak a]_\mathfrak n \in \pi(H),\\ 
0 & \text{otherwise}.
\end{cases}
$$
\end{lem}
\begin{proof}
Let $\Theta_\mathfrak a = \{\theta \in \widehat{G/H} \mid \theta^*(\mathfrak a) \neq 0\} = \{\theta \in \widehat{G/H} \mid (\mathfrak f_{\theta^*}, \mathfrak a) = 1\}$. This set is naturally in bijection with the group $X$ of characters of $\Cl_{\mathfrak n}(K)/\pi(H)$. We obtain
\begin{alignat*}{1}
\sum_{\theta \in \widehat{G/H}} \theta^*(\mathfrak a) &= \sum_{\theta \in \Theta_{\mathfrak a}} \theta^*(\mathfrak a)
= \sum_{\nu \in X} \nu([\mathfrak a]_\mathfrak n) 
= \begin{cases}
[\Cl_{\mathfrak n}(K):\pi(H)] & \text{if } [\mathfrak a]_\mathfrak n \in \pi(H),\\ 
0 & \text{otherwise}.
\end{cases}
\end{alignat*}
\end{proof}

\begin{lem}\label{lemmaSumIntegrals}
For any $0 < a < 1$, we have
\begin{alignat*}{1}
\mathscr S_\mathfrak m(x) + \mathscr S_H(x) & = \frac{-1}{[G:H]} \sum_{\theta \in \widehat{G/H}} I(x,\theta),
\end{alignat*}
where
\begin{alignat*}{1}
\mathscr S_H(x) &= \sum_{\substack{N(\mathfrak a) < x \\ [\mathfrak a]_\mathfrak m  \in H}} \left(1 - \chi(\mathfrak a)\right)P(\mathfrak a,x),\\
\mathscr S_\mathfrak m(x) & = \frac{1}{[G:H]}\sum_ {\theta \in \widehat{G/H}}\sum_{\substack{N(\mathfrak a) < x \\ (\mathfrak a,\mathfrak m) \neq 1}} \left(\theta^*(\mathfrak a) - \chi_\theta(\mathfrak a)\right)P(\mathfrak a,x),\text{ and}\\
I(x,\theta) &= \frac{1}{2\pi i} \int_{2-i\infty}^{2+i\infty}\frac{x^s}{(s+a)^2} \left( \frac{L'_{\theta}}{L_{\theta}} - \frac{L'_{\chi_\theta}}{L_{\chi_\theta}}\right)(s)ds.
\end{alignat*}
\end{lem}

\begin{proof}
From Lemma~\ref{lemma:sumTheta}, for any ray class character $\eta$, we have
\begin{alignat*}{1}
\sum_{\substack{N(\mathfrak a) < x \\ [\mathfrak a]_\mathfrak m  \in H}} \eta(\mathfrak a)P(\mathfrak a,x)
&=  \sum_{\substack{N(\mathfrak a) < x \\ (\mathfrak a, \mathfrak m) = 1}} \frac{\sum_{\theta \in \widehat{G/H}} \theta^*(\mathfrak a)}{[G:H]} \eta(\mathfrak a)P(\mathfrak a,x)
=  \frac{1}{[G:H]} \sum_{\theta \in \widehat{G/H}}\sum_{\substack{N(\mathfrak a) < x \\ (\mathfrak a, \mathfrak m) = 1}}\eta_\theta(\mathfrak a)P(\mathfrak a,x).
\end{alignat*}
Subtracting two instances of this equality, for $\eta = 1$ and $\eta = \chi$, we get
\[ \mathscr S_H(x) = \frac{1}{[G:H]}\sum_ {\theta \in \widehat{G/H}}\sum_{N(\mathfrak a) < x} \left(\theta^*(\mathfrak a) - \chi_\theta(\mathfrak a)\right)P(\mathfrak a,x) - \mathscr S_\mathfrak m(x),\]
and conclude by applying Lemma~\ref{lemmaBachInt}.
\end{proof}

\begin{lem}\label{lemmaSumResidues}
For $0 < a < 1$, and with the notation from Lemma~\ref{lemmaSumIntegrals},
\begin{alignat*}{1}
\frac{x}{(a+1)^2} &= [G:H] ( \mathscr S_H(x) +  \mathscr S_\mathfrak m(x)) + \sum_{\theta \in \widehat{G/H}}\left( I_{1/2}(x,\theta) + I_0(x,\theta)  + I_-(x,\theta) \right)
\end{alignat*}
where
\begin{alignat*}{1}
I_-(x,\theta) =\ &(\beta_{\chi_\theta}-\beta_\theta)\sum_{k = 2}^\infty\frac{(-1)^k}{(a-k)^2x^k},\\
I_{1/2}(x,\theta) =\ &\sum_{\rho \in R_{\theta}}\frac{x^\rho}{(\rho + a)^2} - \sum_{\rho \in R_{\chi_\theta}}\frac{x^\rho}{(\rho + a)^2},\text{ and}\\
I_0(x,\theta) =\ & \frac{\log x}{x^a}\left(\frac{L'_{\theta}}{L_{\theta}} - \frac{L'_{\chi_\theta}}{L_{\chi_\theta}}\right)(-a) + \frac{1}{x^a}\left(\frac{L'_{\theta}}{L_{\theta}} - \frac{L'_{\chi_\theta}}{L_{\chi_\theta}}\right)'(-a)\\
& + (\beta_{\chi_\theta}-\beta_\theta)\left(\frac{1}{a^2} - \frac{1}{x(a-1)^2}\right) -\frac{\delta(\theta)}{a^2}.
\end{alignat*}
Recall that for any character $\eta$, $R_\eta$ is the set of zeros of $L_\eta$ on the strip ${0 < \real(s) < 1}$.
\end{lem}
\begin{proof}
This lemma is an analogue of~\cite[Lem. 4.4]{Ba90}. Evaluating each integral $I(x,\theta)$ by residue using Table~\ref{tab:residues} yields
\[I(x,\theta) =  I_{1/2}(x,\theta) + I_0(x,\theta) +  I_-(x,\theta) - \frac{\delta(\theta) x}{(a+1)^2}.\]
The residue calculations can be justified as in the proof of~\cite[Th. 28]{LO77}. The result follows from Lemma~\ref{lemmaSumIntegrals}.
\end{proof}

\subsection{Explicit estimates.} This section adopts the notation from Lemma~\ref{lemmaSumIntegrals} and Lemma~\ref{lemmaSumResidues}. The remainder of the proof consists in evaluating each term in the formula of Lemma~\ref{lemmaSumResidues}. More precisely, we bound the quantities
\begin{enumerate}
\item  $I_{1/2}$ in Lemma~\ref{lemmaBound12},
\item $I_0$ in Lemma~\ref{lemma_boundI0},
\item $\mathscr S_\mathfrak m$ in Lemma~\ref{lemmaBach57},
\item $\mathscr S_H$ in Lemma~\ref{lemma:BoundSH}
\end{enumerate}
Remains the quantity $I_-$, which is easy to bound thanks to~\cite[Lem. 5.1]{Ba90}. All these estimates are combined in Lemma~\ref{lemmaMainLemma}.
Let \[\mathscr R(a,\chi) = \sum_{\theta \in \widehat{G/H}}\left(\sum_{\rho \in R_{\theta}}\frac{1}{|\rho + a|^2} + \sum_{\rho \in R_{\chi_\theta}}\frac{1}{|\rho + a|^2}\right).\]
We bound that quantity in Lemma~\ref{lemmaBoundR}, but first, we need the following lemma.

\begin{lem}\label{lem:negativeAtS}
For $\real(s) > 1$, we have
$$\sum_{\theta \in \widehat{G/H}}\left(\frac{L'_{\theta}}{L_{\theta}} + \frac{L'_{\chi_\theta}}{L_{\chi_\theta}}\right)(s) \leq 0.$$
\end{lem}
\begin{proof}
Equation~\eqref{eq:formulaLogDerivative} yields
\begin{alignat*}{1}
 \sum_{\theta \in \widehat{G/H}}\left(\frac{L'_{\theta}}{L_{\theta}} + \frac{L'_{\chi_\theta}}{L_{\chi_\theta}}\right)(s)
 & = - \sum_{\theta \in \widehat{G/H}}  \sum_{\mathfrak a}\frac{\Lambda(\mathfrak a)(\chi_\theta(\mathfrak a)+\theta^*(\mathfrak a))}{N(\mathfrak a)^s}\\
 & = -   \sum_{\mathfrak a}\frac{\Lambda(\mathfrak a)}{N(\mathfrak a)^s}\sum_{\theta \in \widehat{G/H}} (\chi_\theta(\mathfrak a)+\theta^*(\mathfrak a)).
\end{alignat*}
Fix an ideal $\mathfrak a$. If $\chi_\theta(\mathfrak a) = 0$ for all $\theta$, Lemma~\ref{lemma:sumTheta} implies that $$\sum_{\theta \in \widehat{G/H}} (\chi_\theta(\mathfrak a)+\theta^*(\mathfrak a)) \geq 0.$$ 
Now suppose that there exists an $\eta \in \widehat{G/H}$ such that $\chi_\eta(\mathfrak a) \neq 0$.
The fact that any given character is induced by a unique primitive character implies that for any $\theta \in \widehat{G/H}$, we have $\chi_{\theta}(\mathfrak a) = \chi_\eta(\mathfrak a)\left( \theta \eta^{-1} \right)^*(\mathfrak a)$. Indeed, if $\left( \theta \eta^{-1} \right)^*(\mathfrak a) \neq 0$, the equality follows from the fact that $\chi_\theta$ is the primitive character inducing $\chi_\eta \cdot\left( \theta \eta^{-1} \right)^*$, and if $\left( \theta \eta^{-1} \right)^*(\mathfrak a) = 0$, then one must have $\chi_{\theta}(\mathfrak a) = 0$ because $\left( \theta \eta^{-1} \right)^*$ is the primitive character inducing $\chi_\theta/\chi_\eta$.
We deduce that
\begin{alignat*}{1}\sum_{\theta \in \widehat{G/H}} (\chi_\theta(\mathfrak a)+\theta^*(\mathfrak a))
&= \chi_\eta(\mathfrak a)\sum_{\theta \in \widehat{G/H}} \left(\frac \theta \eta \right)^*(\mathfrak a)+\sum_{\theta \in \widehat{G/H}}\theta^*(\mathfrak a)
= (\chi_\eta(\mathfrak a)+1)\sum_{\theta \in \widehat{G/H}}\theta^*(\mathfrak a),
\end{alignat*}
whose real part is non-negative (using again Lemma~\ref{lemma:sumTheta}).
\end{proof}

\begin{lem}[ERH]\label{lemmaBoundR}
Let $0<a<1$. The sum $\mathscr R(a,\chi)$ is at most
\begin{alignat*}{1}
\frac{2[G:H]}{2a+1}\bigg(&\log(\Delta N(\mathfrak m_0)) + n(\psi(a+1) - \log(2\pi)) \\
&- \frac{|\mathfrak m_\infty|}{2}\left(\psi\left(\frac{a+1}{2}\right) - \psi\left(\frac{a+2}{2}\right)\right)\bigg) + \frac{2}{2a+1}\left(\frac{1}{a+1}+\frac{1}{a}\right).
\end{alignat*}
\end{lem}

\begin{proof}
Writing $\sigma = 1+a$, we have $\frac{2a + 1}{|\rho + a|^2} = \frac{1}{\sigma - \rho} + \frac{1}{\sigma - \bar\rho}$ for any $\real(\rho) = 1/2$ (as observed in~\cite[Lemma 5.5]{Ba90}), so for any ray class character $\eta$
\begin{alignat*}{1}
\sum_{\rho \in R_{\eta}}\frac{1}{|\rho + a|^2} &= \frac{1}{2a+1}\sum_{\rho \in R_{\eta}}\left(\frac{1}{\sigma - \rho} + \frac{1}{\sigma - \bar\rho}\right).
\end{alignat*}
As in~\cite[Lem. 5.1]{LO77}, we get from Equation~\eqref{eq:logdlchi} that
\[\sum_{\rho \in R_{\eta}}\left(\frac{1}{\sigma - \rho} + \frac{1}{\sigma - \bar\rho}\right) = 2\real\frac{L'_\eta}{L_\eta}(\sigma) + \log(\Delta N(\mathfrak f_{\eta})) + 2\delta(\eta)\left(\frac{1}{\sigma}+\frac{1}{\sigma-1}\right) + 2\psi_\eta(\sigma).\]
Then, $\mathscr R(a,\eta)$ is at most
\begin{alignat}{1}\label{eq:tmpBoundR}
\frac{1}{2a+1}\sum_{\theta \in \widehat{G/H}} \bigg( &2\real\left(\frac{L'_\theta}{L_\theta} + \frac{L'_{\chi_\theta}}{L_{\chi_\theta}}\right)(\sigma) + \log(\Delta^2 N(\mathfrak f_\theta \mathfrak f_{\chi_\theta})) \\
\nonumber & + 2\delta(\theta)\left(\frac{1}{\sigma}+\frac{1}{\sigma-1}\right) + 2(\psi_\theta(\sigma) + \psi_{\chi_\theta}(\sigma)) \bigg).
\end{alignat}
From Lemma~\ref{lem:negativeAtS}, we have
$\sum_{\theta \in \widehat{G/H}} \left(\frac{L'_\theta}{L_\theta} + \frac{L'_{\chi_\theta}}{L_{\chi_\theta}}\right)(\sigma) \leq 0$, and the corresponding term can be discarded from the expression in~\eqref{eq:tmpBoundR}.
Also, with $\alpha_{\chi_\theta} = r_1 - \beta_{\chi_\theta}$,
\begin{alignat*}{1}
2\left(\psi_\theta(\sigma) + \psi_{\chi_\theta}(\sigma)\right) 
& = (n+\alpha_{\chi_\theta}-\beta_\theta)\psi\left(\frac{a+1}{2}\right) + (n-\alpha_{\chi_\theta}+\beta_\theta)\psi\left(\frac{a+2}{2}\right) - 2n\log \pi \\
& = 2n(\psi(a+1) - \log(2\pi)) + (\alpha_{\chi_\theta}-\beta_\theta)\left(\psi\left(\frac{a+1}{2}\right) - \psi\left(\frac{a+2}{2}\right)\right)\\
& \leq 2n(\psi(a+1) - \log(2\pi)) - |\mathfrak m_\infty|\left(\psi\left(\frac{a+1}{2}\right) - \psi\left(\frac{a+2}{2}\right)\right),
\end{alignat*}
where the first equality uses the expression~\eqref{eq:logdgamma} and the second one follows from the duplication formula ($\psi(z/2) + \psi((z+1)/2) = 2(\psi(z) - \log 2)$).
\end{proof}
\begin{lem}[ERH]\label{lemmaBound12}
For $0<a<1$ and $x \geq 1$, $\sum_{\theta \in \widehat{G/H}}|I_{1/2}(x,\theta)| \leq \sqrt{x}\cdot\mathscr R(a,\chi)$.
\end{lem}
\begin{proof}
From the ERH, for any ray class character $\eta$, and any zero $\rho \in R_\eta$ of $L_\eta$ on the critical strip, we have $\real(\rho) \leq 1/2$. Therefore $|x^\rho| = |x|^{\real (\rho)} \leq \sqrt{x}$.
\end{proof}

\begin{lem}\label{lemma:diffLsAndDerivative}
For any $s$,
\begin{alignat*}{1}
\left(\frac{L'_{\theta}}{L_{\theta}} - \frac{L'_{\chi_\theta}}{L_{\chi_\theta}}\right)(s)  =&\ \sum_{\rho \in R_{\theta}}\left(\frac{1}{s - \rho} - \frac{1}{2 - \rho}\right) - \sum_{\rho \in R_{\chi_\theta}}\left(\frac{1}{s - \rho} - \frac{1}{2 - \rho}\right)\\
&\ -\frac{\beta_{\chi_\theta}-\beta_\theta}{2}\left(\psi\left(\frac{s}{2}\right) - \psi\left(\frac{s+3}{2}\right) - \psi(1) + \psi\left(\frac{3}{2}\right)\right)\\
&\  - \frac{\beta_{\chi_\theta} - \beta_\theta}{s+1}+\delta(\theta)\left(\frac{3}{2} - \frac{1}{s} - \frac{1}{s-1}\right) + \left(\frac{L'_{\theta}}{L_{\theta}} - \frac{L'_{\chi_\theta}}{L_{\chi_\theta}}\right)(2),
\end{alignat*}
and
\begin{alignat*}{1}
\left(\frac{L'_{\theta}}{L_{\theta}} - \frac{L'_{\chi_\theta}}{L_{\chi_\theta}}\right)'(s) = &\ \sum_{\rho \in R_{\chi_\theta}}\frac{1}{(s - \rho)^2} - \sum_{\rho \in R_{\theta}}\frac{1}{(s - \rho)^2}\\
&\ -\frac{\beta_{\chi_\theta}-\beta_\theta}{4}\left(\psi'\left(\frac{s}{2}\right) - \psi'\left(\frac{s+3}{2}\right)\right)\\
&\ + \frac{\beta_{\chi_\theta} - \beta_\theta}{(s+1)^2} +\delta(\theta)\left(\frac{1}{s^2} + \frac{1}{(s-1)^2}\right).
\end{alignat*}
\end{lem}
\begin{proof}
This is essentially the same proof as~\cite[Lem. 5.2]{Ba90}, with an additional use of the recurrence relations $\psi(z) = \psi(z+1) - 1/z$ and $\psi'(z) = \psi'(z+1) + 1/z^2$.
\end{proof}

\begin{lem}[ERH]\label{lemma_boundI0}
Let $0 < a < 1$ and $x \geq 1$. Then,
\begin{alignat*}{1}
\sum_{\theta \in \widehat{G/H}} I_0(x,\theta) \leq &\ \frac{(2+a)\log x + 1}{x^a}\cdot \mathscr R(a,\chi)  + \frac{[G:H]|\mathfrak m_\infty|}{a^2} - \frac{1}{a^2}\\
&\ + \frac{\log x}{x^a}\left(\frac{3}{2} + \frac{1}{a} + \frac{1}{a+1} \right) + \frac{1}{x^a}\left(\frac{1}{a^2} + \frac{1}{(a+1)^2}\right)\\
&\ + \frac{[G:H]|\mathfrak m_\infty|}{x}\left(\frac{1}{(1-a)^2} - \frac{\log x}{(a-1)x^{a-1}} - \frac{1}{(a-1)^2x^{a-1}} \right).
\end{alignat*}
\end{lem}

\begin{proof}

For any $0 < a < 1$, Lemma~\ref{lemma:diffLsAndDerivative} implies that
\begin{alignat*}{1}
\sum_{\theta \in \widehat{G/H}} \left(\frac{L'_{\theta}}{L_{\theta}} - \frac{L'_{\chi_\theta}}{L_{\chi_\theta}}\right)(-a) \leq &\ (2+a) \cdot \mathscr R(a,\chi) +\frac{3}{2} + \frac{1}{a} + \frac{1}{a+1} - \sum_{\theta \in \widehat{G/H}}\frac{\beta_{\chi_\theta}-\beta_\theta}{1-a},
\end{alignat*}
and
\begin{alignat*}{1}
\sum_{\theta \in \widehat{G/H}} \left(\frac{L'_{\theta}}{L_{\theta}} - \frac{L'_{\chi_\theta}}{L_{\chi_\theta}}\right)'(-a)  \leq &\ \mathscr R(a,\chi) +\frac{1}{a^2} + \frac{1}{(a+1)^2} + \sum_{\theta \in \widehat{G/H}}\frac{\beta_{\chi_\theta} - \beta_\theta}{(1-a)^2}.
\end{alignat*}
We used the facts that
$\psi\left(\frac{-a}{2}\right) - \psi\left(\frac{3-a}{2}\right) - \psi(1) + \psi\left(\frac{3}{2}\right) \geq 0,$
and
$\psi'\left(\frac{-a}{2}\right) - \psi'\left(\frac{3-a}{2}\right) \geq 0,$
which are easily derived from the recurrence relations $\psi(z) = \psi(z+1) - 1/z$ and $\psi'(z) = \psi'(z+1) + 1/z^2$, and the monotonicity of $\psi$ and $\psi'$.
From~\cite[Lem. 5.3]{Ba90}, for any $0 < a < 1$, we have
$
\left(\frac{\log x}{(a-1)x^{a-1}} + \frac{1}{(a-1)^2x^{a-1}} -  \frac{1}{(1-a)^2} \right) \leq 0,
$
therefore
\begin{alignat*}{1}
& \sum_{\theta \in \widehat{G/H}}\frac{\beta_{\chi_\theta} - \beta_\theta}{x}\left(\frac{\log x}{(a-1)x^{a-1}} + \frac{1}{(a-1)^2x^{a-1}} -  \frac{1}{(1-a)^2} \right)
\\  & \leq \frac{[G:H]|\mathfrak m_\infty|}{x}\left(\frac{1}{(1-a)^2} - \frac{\log x}{(a-1)x^{a-1}} - \frac{1}{(a-1)^2x^{a-1}} \right).
\end{alignat*}
The result follows by applying these estimates to $I_0(x,\theta)$ (as defined in Lemma~\ref{lemmaSumResidues}).
\end{proof}

\begin{lem}\label{lemmaBach57}
For any $0 < a < 1$,
\[\mathscr S_\mathfrak m(x) \leq \frac{2\log x}{ea}\omega(\mathfrak m_0) \leq \frac{2\log x}{ea\log 2}\log(N(\mathfrak m_0)),\]
where $\omega(\mathfrak m_0)$ is the number of distinct prime ideals dividing $\mathfrak m_0$.
\end{lem}

\begin{proof}
We have
\begin{alignat*}{1}
\mathscr S_\mathfrak m(x) & = \frac{1}{[G:H]}\sum_{\substack{N(\mathfrak a) < x \\ (\mathfrak a,\mathfrak m) \neq 1}} \left({\sum_ {\theta \in \widehat{G/H}} \left(\theta^*(\mathfrak a) - \chi_\theta(\mathfrak a)\right)}\right) P(\mathfrak a,x) \leq \sum_{\substack{N(\mathfrak a) < x \\ (\mathfrak a,\mathfrak m) \neq 1}} 2P(\mathfrak a,x),
\end{alignat*}
and the result follows from \cite[Lem. 5.7]{Ba90}.
\end{proof}

\begin{lem}[ERH]\label{lemmaMainLemma}
For any $0 < a < 1$, the fraction ${\sqrt{x}}/{(a+1)^2}$ is at most
\begin{alignat*}{1}
 [G:H]\bigg(s_1(x)\log(\Delta N(\mathfrak m_0)) + s_5(x)n +  s_4(x)|\mathfrak m_\infty| + s_3(x)\omega(\mathfrak m_0) + \frac{\mathscr S_H(x)}{\sqrt{x}}\bigg) + s_2(x),
\end{alignat*}
where
\begin{alignat*}{1}
s_1(x) =&\ \frac{2}{2a+1}\left(1 + \frac{(2+a)\log x +1}{x^{a + 1/2}}\right),\\
s_2(x) =&\ s_1(x)\left(\frac{1}{a} + \frac{1}{a+1}\right) + \frac{\log x}{x^{a + 1/2}}\left(\frac{3}{2} + \frac{1}{a} + \frac{1}{a+1}\right)
+\frac{1}{x^{a + 1/2}}\left(\frac{1}{a^2} + \frac{1}{(a+1)^2} \right),\\
s_3(x) =&\ \frac{2\log x}{e a \sqrt{x}},\\
s_4(x) =
&\ \frac{1}{(a-2)^2x^{5/2}} -\frac{s_1(x)}{2}\left(\psi\left(\frac{a+1}{2}\right) - \psi\left(\frac{a+2}{2}\right)\right) + \frac{1}{a^2\sqrt{x}}\\
&\ +\frac{1}{x^{3/2}}\left(\frac{1}{(1-a)^2} - \frac{\log x}{(a-1)x^{a-1}} - \frac{1}{(a-1)^2x^{a-1}} \right),\\
s_5(x) =&\ s_1(x)(\psi(a+1) - \log(2\pi)).
\end{alignat*}
\end{lem}

\begin{proof}
As in~\cite[Lem. 5.1]{Ba90}, we have $0 \leq \sum_{k=2}^\infty \frac{(-1)^k}{(a-k)^2x^k} \leq \frac{1}{(a-2)^2x^2}$.
We deduce that $I_-(x,\theta) \leq \frac{|\beta_{\chi_\theta} - \beta_\theta|}{(a-2)^2x^2} \leq \frac{|\mathfrak m_\infty|}{(a-2)^2x^2}$.
Together with Lemma~\ref{lemmaBound12}, the bound from Lemma~\ref{lemmaSumResidues} becomes
\begin{alignat*}{1}
\frac{\sqrt{x}}{(a+1)^2} &\leq \frac{[G:H]|\mathfrak m_\infty|}{(a-2)^2x^{5/2}} + \mathscr R(a,\chi) + \frac{1}{\sqrt{x}}\sum_{\theta \in \widehat{G/H}} I_0(x,\theta) + [G:H]   \frac{\mathscr S_H(x) + \mathscr S_\mathfrak m(x)}{\sqrt{x}}.
\end{alignat*}
The result then follows from Lemma~\ref{lemmaBoundR}, Lemma~\ref{lemma_boundI0} and Lemma~\ref{lemmaBach57}.
\end{proof}

\begin{lem}\label{lemma:BoundSH}
Suppose that $\chi(\mathfrak p) = 1$ for all prime ideals $\mathfrak p$ such that $N(\mathfrak p) < x$, ${[\mathfrak p]_\mathfrak m  \in H}$, and $\deg(\mathfrak p) = 1$. Then, 
for any $0 < a < 1$,
$$\mathscr S_H(x) \leq \frac{2n}{ea}\sum_{\substack{m < \sqrt x}} \Lambda(m).$$
\end{lem}
\begin{proof}
We start as in~\cite[Lem. 5.7]{Ba90} by observing that when $t \geq 1$, the function $t^{-a}\log t$ is bounded above by $1/ea$. We deduce
\begin{equation}\label{eq:boundSH}
\mathscr S_H(x) = \sum_{\substack{N(\mathfrak a) < x \\ [\mathfrak a]_\mathfrak m  \in H}} \left(1 - \chi(\mathfrak a)\right)P(\mathfrak a,x) \leq \frac{2}{ea} \sum_{\substack{N(\mathfrak a) < x \\ [\mathfrak a]_\mathfrak m  \in H\\ \chi(\mathfrak a) \neq 1}} \Lambda(\mathfrak a).
\end{equation}
Fix a prime ideal $\mathfrak p$ (above a rational prime $p$) of norm smaller than $x$ and consider the contribution of its powers to the above sum.
First suppose that $\deg(\mathfrak p) > 1$. Then,
$$\sum_{\substack{N(\mathfrak p^k) < x \\ [\mathfrak p^k]_\mathfrak m  \in H\\ \chi(\mathfrak p^k) \neq 1}} \Lambda(\mathfrak p^k)
\leq \sum_{\substack{N(\mathfrak p^k) < x}} \deg(\mathfrak p)\Lambda(p^k) \leq  \deg(\mathfrak p)\sum_{\substack{ p^k < \sqrt{x}}}\Lambda(p^k).$$
Now suppose that $\deg(\mathfrak p) = 1$, and let $\ell$ be the smallest integer such that $[\mathfrak p^\ell]_\mathfrak m \in H$. 
If $\ell = 1$, then $\chi(\mathfrak p^k) = 1$ for any integer $k$, so the contribution of $\mathfrak p$ is zero. Suppose that $\ell \geq 2$.
Then,
$$\sum_{\substack{N(\mathfrak p^k) < x \\ [\mathfrak p^k]_\mathfrak m  \in H\\ \chi(\mathfrak p^k) \neq 1}} \Lambda(\mathfrak p^k)
\leq \sum_{\substack{N(\mathfrak p^{k\ell}) < x}} \Lambda(\mathfrak p^{k\ell}) \leq \deg(\mathfrak p)\sum_{\substack{p^{k} < \sqrt x}} \Lambda(p^{k}).$$
Summing over all rational primes $p$ and ideals $\mathfrak p$ above $p$, we obtain
\begin{alignat*}{1}
\sum_{p}\sum_{\mathfrak p \mid p} \sum_{\substack{N(\mathfrak p^k) < x \\ [\mathfrak p^k]_\mathfrak m  \in H\\ \chi(\mathfrak p^k) \neq 1}} \Lambda(\mathfrak p^k)
\leq \sum_{p}\sum_{\mathfrak p \mid p}\deg(\mathfrak p)\sum_{\substack{p^{k} < \sqrt x}} \Lambda(p^{k})
\leq n\sum_{\substack{m < \sqrt x}} \Lambda(m).
\end{alignat*}
We conclude by applying this inequality to Equation~\eqref{eq:boundSH}.
\end{proof}

\begin{lem}\label{lemma:limit}
For any $x>0$,
$$\lim_{a \rightarrow 1}\left(\frac{1}{(1-a)^2} - \frac{\log x}{(a-1)x^{a-1}} - \frac{1}{(a-1)^2x^{a-1}} \right) = \frac{(\log x)^2}{2}.$$
\end{lem}
\begin{proof}
A simple application of l'H\^opital's rule yields
\begin{alignat*}{1}
&\lim_{a \rightarrow 1}\left(\frac{1}{(1-a)^2} - \frac{\log x}{(a-1)x^{a-1}} - \frac{1}{(a-1)^2x^{a-1}} \right)\\
&= \lim_{b \rightarrow 0}\left(\frac{x^b - b\log x - 1}{b^2x^b} \right) = \lim_{b \rightarrow 0}\left(\frac{x^b\log x - \log x}{bx^b(b\log(x) + 2)} \right) \\
&= \lim_{b \rightarrow 0}\left(\frac{(\log x)^2}{b^2(\log x)^2 + 4b\log x + 2} \right)  = \frac{(\log x)^2}{2}.
\end{alignat*}
\end{proof}

\subsection{Proof of Theorem~\ref{mainTheorem}}
Let $x$ be the norm of the smallest prime ideal $\mathfrak p$ such that $[\mathfrak p]_\mathfrak m \in H$, $\deg(\mathfrak p) = 1$ and $\chi(\mathfrak p) \neq 1$. 
First suppose that $x \leq 95$, and consider the quantity
\begin{equation*}
B = \big([G:H]\left(2.71\log(\Delta N(\mathfrak m_0)) +  1.29|\mathfrak m_\infty| + 1.38\omega(\mathfrak m_0)\right) + 4.13\big)^2.
\end{equation*}
We want to show that $x \leq B$.
\subsubsection*{Suppose $n = 1$.} For the ray class group $G$ not to be trivial, one must have either $|\mathfrak m_\infty| = 1$ and $N(\mathfrak m_0) \geq 3$, in which case
\begin{equation*}
B \geq \big(2.71\log(3) +  1.29 + 1.38 + 4.13\big)^2 = 95.59\dots \geq x,
\end{equation*}
or $|\mathfrak m_\infty| = 0$ and $N(\mathfrak m_0) \geq 5$, in which case
\begin{equation*}
B \geq \big(2.71\log(5) + 1.38 + 4.13\big)^2 = 97.44\dots \geq x.
\end{equation*}

\subsubsection*{Suppose $n = 2$.}
Suppose that $\Delta N(\mathfrak m_0) \geq 8$. Then
\begin{equation*}
B \geq \left(2.71\log(8) + 4.13\right)^2 = 95.36\dots \geq x.
\end{equation*}
Now, one must investigate the cases where $\Delta N(\mathfrak m_0) \leq 7$.
All quadratic fields with a discriminant of absolute value at most 7 have a trivial (narrow) class group.
Therefore, one must have $N (\mathfrak m_0) \geq 2$. There is only one quadratic field of discriminant of absolute value at most $3$, namely $\Q(\sqrt{-3})$. It has discriminant of absolute value $3$ and no ideal of norm $2$, so the condition $\Delta N(\mathfrak m_0) \leq 7$ is impossible.

\subsubsection*{Suppose $n > 2$.}  From~\cite[Lem. 7.1]{Ba90}, we get
$$\log(\Delta N(\mathfrak f)) \geq n(\log(2\pi) - \psi(2)) - \frac{3}{2} \geq 2.74,$$
and we deduce
\begin{equation*}
B \geq \left(2.71\cdot2.74 + 4.13\right)^2 = 133.52\dots \geq x.
\end{equation*}
It remains to consider the case $x > 95$.
From Lemma~\ref{lemma:BoundSH} and~\cite[Th.~12]{RS62},
\begin{alignat*}{1}
\mathscr S_H(x) \leq \frac{2n}{ea}\sum_{\substack{m < \sqrt x}} \Lambda(m) \leq \frac{2nC\sqrt{x}}{ea},
\end{alignat*}
where $C = 1.03883$. 
We now apply Lemma~\ref{lemmaMainLemma} with $a\rightarrow 1$. From Lemma~\ref{lemma:limit} (applied to the term $s_4$), and the facts that for $x \geq 95$, $\left(s_5(x) + \frac{2C}{ea}\right)$ is negative, and $s_1,s_2,s_3$ and $s_4$ are decreasing, we get
\begin{alignat*}{1}
x &\leq 2^4\big([G:H]\left(s_1(95)\log(\Delta N(\mathfrak m_0)) +  s_4(95)|\mathfrak m_\infty| + s_3(95)\omega(\mathfrak m_0)\right) + s_2(95)\big)^2\\
&\leq \big([G:H]\left(2.71\log(\Delta N(\mathfrak m_0)) +  1.29|\mathfrak m_\infty| + 1.38\omega(\mathfrak m_0)\right) + 4.13\big)^2,
\end{alignat*}
which proves the theorem.
\qed

\section{Consequences}\label{sec:consequences}

With Theorem~\ref{mainTheorem} at hands, we can now derive a few important consequences. The first of them, Theorem~\ref{thmBoundRayClass}, asserts that a subgroup $H$ of the ray class group $\Cl_{\mathfrak m}(K)$ is always generated by ideals of bounded prime norm.

\subsection{Proof of Theorem~\ref{thmBoundRayClass}}
Recall that $K$ is a number field, $\Delta$ is the absolute value of the discriminant of $K$, and $\mathfrak m$ is a modulus of $K$, with finite part $\mathfrak m_0$ and infinite part $\mathfrak m_\infty$.
Also, $\mathfrak h$ is an ideal in $K$, and $H$ is a non-trivial subgroup of the ray class group $\Cl_\mathfrak m(K)$.
Let 
\begin{alignat*}{1}
B &= \left([G:H]\left(2.71\log(\Delta N(\mathfrak h\mathfrak m_0)) +  1.29|\mathfrak m_\infty| + 1.38\omega(\mathfrak h\mathfrak m_0)\right) + 4.13\right)^2,\\
\mathscr N &= \{\mathfrak p \in \mathscr I_\mathfrak m(K) \mid \mathfrak p \text{ is prime}, (\mathfrak p, \mathfrak h) = 1, [\mathfrak p]_\mathfrak m \in H, \deg(\mathfrak p) = 1 \text{ and } N(\mathfrak p) < B\},
\end{alignat*}
 and $N$ the subgroup of $H$ generated by $\mathscr N$. By contradiction, suppose $N \neq H$. Then, there is a non-trivial character of $H$ that is trivial on $N$. Since $G$ is abelian, this character on $H$ extends to a character on $G$, thereby defining a ray class character $\chi$ modulo $\mathfrak m$ that is not trivial on $H$.
From Theorem~\ref{mainTheorem}, 
there is a prime ideal $\mathfrak p \in \mathscr I_{\mathfrak h \mathfrak m}(K)$ such that $[\mathfrak p]_\mathfrak m \in H$, $\chi(\mathfrak p) \neq 1$, $\deg( \mathfrak p) = 1$ and
$N(\mathfrak p) \leq B.$
All these conditions imply that $\mathfrak p \in \mathscr N \subseteq N$, whence $\chi(\mathfrak p) = 1$, a contradiction.\qed\\

The next consequence, Theorem~\ref{thmBoundIntegers}, is a specialization of Theorem~\ref{thmBoundRayClass} to the field of rational numbers, and asserts that a subgroup $H$ of a group of the form $(\Z/m \Z)^\times$ is generated by prime numbers bounded polynomially in the subgroup index and $\log(m)$.

\subsection{Proof of Theorem~\ref{thmBoundIntegers}}
Recall that $m$ is a positive integer, and $H$ is a non-trivial subgroup of $G = (\Z/m\Z)^\times$.
Let $\mathfrak m = \mathfrak m_0 \mathfrak m_\infty$ where $\mathfrak m_0 = m\Z$ and $\mathfrak m_\infty$ is the real embedding of $\Q$. Then, $\Cl_\mathfrak m(\Q)$ is isomorphic to $G = (\Z/m\Z)^\times$.  An isomorphism is given by the map sending the class of $a\Z$ to $a \text{ mod } m$. The subgroup $H$ of $(\Z/m\Z)^\times$ corresponds to a subgroup $H'$ of $\Cl_\mathfrak m(\Q)$ through this isomorphism. From Theorem~\ref{thmBoundRayClass}, $H'$ is generated by prime numbers smaller than
$$B = \left([G:H]\left(2.71\log(m) +  1.29 + 1.38\omega(m)\right) + 4.13\right)^2,$$
and so is $H$. If $H$ is the full group, then the theorem follows from~\cite[Th. 3]{Ba90}; and for $m \leq 11000$, the result is easy to check by an exhaustive computation. So we can assume that $m/|H| \geq 2$ and $m > 11000$. From~\cite[Lem. 6.4]{Ba90},
\begin{alignat*}{1}
\frac{\omega(m)}{\log m} &\leq \frac{\mathrm{li}(\log m) + 0.12\sqrt{\log m}}{\log m} 
\leq  \frac{\mathrm{li}(\log 11000) + 0.12\sqrt{\log 11000}}{\log 11000} \leq 0.67,
\end{alignat*}
where $\mathrm{li}$ is the logarithmic integral function. We get
$$B \leq \left([G:H]\log(m)\left(2.71 +  \frac{1.29 + 4.13/2}{\log 11000}  + 1.38\cdot 0.67\right)\right)^2,$$
and we conclude by computing the constant.
\qed\\

The third consequence is a bound on the degrees of the cyclic isogenies required to connect all isogenous principally polarizable abelian varieties over a finite field sharing the same endomorphism ring.

\subsection{Proof of Theorem~\ref{thmBoundIsogenies}}
Recall that $\mathscr A$ is a principally polarized, absolutely simple, ordinary abelian variety over a finite field $\F_q$, with endomorphism algebra $K$ and endomorphism ring isomorphic to an order $\cO$ in $K$. The field $K_0$ is the maximal real subfield of $K$, and $\mathfrak f$ is the conductor of $\cO$. For any $B > 0$, $\mathscr G(B)$ is the isogeny graph whose vertices are the principally polarizable varieties isogenous to $\mathscr A$ and with the same endomorphism ring, and whose edges are isogenies connecting them, of prime degree (therefore cyclic) smaller than $B$.
By the theory of complex multiplication, the graph $\mathscr G(B)$ is isomorphic to the Cayley graph of $$\mathscr P(\cO) = \ker(\Cl(\cO) \rightarrow \Cl^+(\cO\cap K_0))$$ with set of generators the classes of  ideals of prime norm smaller than $B$ (see~\cite[Sec. 2.5]{JW16} for a detailed discussion on this isomorphism).
Let $g \geq 2$ be the dimension of $\mathscr A$, and $n = 2g$ the degree of its endomorphism algebra $K$.
The natural map $\pi: \Cl_\mathfrak f(K) \rightarrow \Cl(\cO)$ is a surjection (see for instance~\cite[Sec. 2.2]{JW16}), so it is sufficient to find a generating set for $H = \pi^{-1}(\mathscr P(\cO))$.
From~\cite[Lem. 2.1]{JW16}, we have the inequality $$[\Cl_\mathfrak f(K):H] \leq [\Cl(\cO):\mathscr P(\cO)] \leq h_{\cO_0}^+.$$
From Theorem~\ref{thmBoundRayClass}, $\mathscr G(B)$ is connected for
\begin{equation}\label{eq:boundProofIsogeny}
B = \left(2.71 + 1.38\frac{\omega(\mathfrak f)}{\log(\Delta N(\mathfrak f))} + \frac{4.13}{\log(\Delta N(\mathfrak f))}\right)^2\left(h_{\cO_0}^+\log(\Delta N(\mathfrak f))\right)^2,
\end{equation}
and it remains to show that the constant factor in this expression is at most $26$.
First, we need a lower bound on the quantity $\log(\Delta N(\mathfrak f))$. From \cite[Tab. 3]{Odlyzko90}, if $n = 4$, $\log(\Delta N(\mathfrak f)) \geq 4\log(3.263) \geq 4.73$ (this result assumes ERH). For $n \geq 6$,~\cite[Lem. 7.1]{Ba90} implies
$$\log(\Delta N(\mathfrak f)) \geq n(\log(2\pi) - \psi(2)) - \frac{3}{2} \geq 6.99.$$
Therefore for any degree $n \geq 4$, we have $\log(\Delta N(\mathfrak f)) \geq 4.73$.
Now, for $n = 2$, smaller values of $\log(\Delta N(\mathfrak f))$ are possible. One can easily check that the constant factor in the expression~\eqref{eq:boundProofIsogeny} is at most $26$ for all pairs $(\Delta, N(\mathfrak f))$ such that $\log(\Delta N(\mathfrak f)) < 4.73$ by an exhaustive computation. There are however five exceptions: when the field is $\Q(\sqrt{-1})$, and $N(\mathfrak f) \in \{1,2\}$, when the field is $\Q(\sqrt{-3})$, and $N(\mathfrak f) \in \{1,3\}$, and when the field is $\Q(\sqrt{5})$, and $N(\mathfrak f) = 1$. 
Since $\mathfrak f$ is the conductor of an order in a quadratic field, it is generated by an integer, so $N(\mathfrak f)$ must be a square. This discards the cases $N(\mathfrak f) \in \{2,3\}$.
When $N(\mathfrak f) = 1$, the order $\mathcal O$ is the ring of integers, which has a trivial (narrow) class group for $\Q(\sqrt{-1})$, $\Q(\sqrt{-3})$ and $\Q(\sqrt{5})$.

Then, irrespective of the value of $n$, we can assume in the rest of the proof that $\log(\Delta N(\mathfrak f)) \geq 4.73$.
If $\omega(\mathfrak f) \leq 5$, then $$\frac{\omega(\mathfrak f)}{\log(\Delta N(\mathfrak f))} \leq \frac{5}{4.73} \leq 1.06.$$
 If $\omega(\mathfrak f) > 5$, then $N(\mathfrak f) \geq 2\cdot3\cdot5\cdot7\cdot11\cdot13^{\omega(\mathfrak f)-5}$, and
$$\frac{\omega(\mathfrak f)}{\log(\Delta N(\mathfrak f))} \leq \frac{\omega(\mathfrak f)}{\log(2\cdot3\cdot5\cdot7\cdot11\cdot13^{\omega(\mathfrak f)-5})} \leq \frac{5}{\log(2\cdot3\cdot5\cdot7\cdot11)} + \frac{1}{\log(13)} \leq 1.06.$$
Then, $$\left(2.71 + \frac{1.38 \cdot \omega(\mathfrak f)}{\log(\Delta N(\mathfrak f))} + \frac{4.13}{\log(\Delta N(\mathfrak f))}\right)^2 \leq (2.71 + 1.38\cdot1.06 + 4.13/4.73)^2 \leq 26,$$
which concludes the proof.
\qed

\section*{Acknowledgements}
The author wishes to thank Arjen K. Lenstra and Rob Granger, as well as the anonymous referees, for their helpful feedback.
Part of this work was supported by the Swiss National Science Foundation under grant number 200021-156420.

\bibliographystyle{amsplain}
\bibliography{biblio-math,biblio-crypto}

\def\cprime{$'$}
\providecommand{\bysame}{\leavevmode\hbox to3em{\hrulefill}\thinspace}
\providecommand{\MR}{\relax\ifhmode\unskip\space\fi MR }
\providecommand{\MRhref}[2]{%
  \href{http://www.ams.org/mathscinet-getitem?mr=#1}{#2}
}
\providecommand{\href}[2]{#2}
\begin{thebibliography}{10}

\bibitem{Ba90}
E.~Bach, \emph{Explicit bounds for primality testing and related problems},
  Mathematics of Computation \textbf{55} (1990), no.~191, 355--380.
  \MR{91m:11096}

\bibitem{BS96}
Eric Bach and Jonathan~P. Sorenson, \emph{Explicit bounds for primes in residue
  classes}, Math. Comput. \textbf{65} (1996), 1717--1735.

\bibitem{BS16}
J.-F. Biasse and F.~Song, \emph{Efficient quantum algorithms for computing
  class groups and solving the principal ideal problem in arbitrary degree
  number fields}, Proceedings of the Twenty-Seventh Annual ACM-SIAM Symposium
  on Discrete Algorithms, SIAM, 2016, pp.~893--902.

\bibitem{CGS14}
P.~Campbell, M.~Groves, and D.~Shepherd, \emph{Soliloquy: A cautionary tale},
  ETSI 2nd Quantum-Safe Crypto Workshop, 2014, Available at
  \url{http://docbox.etsi.org/Workshop/2014/201410_CRYPTO/S07_Systems_and_Attacks/S07_Groves_Annex.pdf}.

\bibitem{CDPR16}
R.~Cramer, L.~Ducas, C.~Peikert, and O.~Regev, \emph{Recovering short
  generators of principal ideals in cyclotomic rings}, pp.~559--585, Springer
  Berlin Heidelberg, Berlin, Heidelberg, 2016.

\bibitem{CDW16}
Ronald Cramer, L{\'e}o Ducas, and Benjamin Wesolowski, \emph{Short
  {S}tickelberger class relations and application to {I}deal-{SVP}}, Advances
  in Cryptology -- EUROCRYPT 2017 (Jean-S{\'e}bastien Coron and Jesper~Buus
  Nielsen, eds.), Springer International Publishing, 2017, pp.~324--348.

\bibitem{GHS02}
S.~D. Galbraith, F.~Hess, and N.~P. Smart, \emph{Extending the {G}{H}{S} {W}eil
  descent attack}, Proceedings of the International Conference on the Theory
  and Applications of Cryptographic Techniques: Advances in Cryptology (London,
  UK), EUROCRYPT '02, Springer-Verlag, 2002, pp.~29--44.

\bibitem{GGH13}
S.~Garg, C.~Gentry, and S.~Halevi, \emph{Candidate multilinear maps from ideal
  lattices}, EUROCRYPT, 2013, pp.~1--17.

\bibitem{JMV09}
D.~Jao, S.~D. Miller, and R.~Venkatesan, \emph{Expander graphs based on
  {G}{R}{H} with an application to elliptic curve cryptography}, J. Number
  Theory \textbf{129} (2009), no.~6, 1491 -- 1504.

\bibitem{JW16}
D.~Jetchev and B.~Wesolowski, \emph{Horizontal isogeny graphs of ordinary
  abelian varieties and the discrete logarithm problem}, Cryptology ePrint
  Archive, Report 2017/053, 2017, \url{http://eprint.iacr.org/2017/053}.

\bibitem{LMO79}
J.C. Lagarias, H.L. Montgomery, and A.M. Odlyzko, \emph{A bound for the least
  prime ideal in the {C}hebotarev density theorem}, Inventiones mathematicae
  \textbf{54} (1979), 271--296 (eng).

\bibitem{LO77}
J.C. Lagarias and A.M. Odlyzko, \emph{Effective versions of the {C}hebotarev
  density theorem}, Algebraic number fields: $L$-functions and Galois
  properties (Proc. Sympos., Univ. Durham, Durham, 1975), Academic Press,
  London, 1977, pp.~409--464.

\bibitem{LLS15}
Y.~Lamzouri, X.~Li, and K.~Soundararajan, \emph{Conditional bounds for the
  least quadratic non-residue and related problems}, Mathematics of Computation
  \textbf{84} (2015), no.~295, 2391--2412.

\bibitem{LSS14}
A.~Langlois, D.~Stehl{\'e}, and R.~Steinfeld, \emph{{GGHLite}: More efficient
  multilinear maps from ideal lattices}, Advances in Cryptology--EUROCRYPT
  2014, Springer, 2014, pp.~239--256.

\bibitem{MN02}
D.~Maisner and E.~Nart, \emph{Abelian surfaces over finite fields as
  jacobians}, Experiment. Math. \textbf{11} (2002), 321–337.

\bibitem{Neukirch99}
J.~Neukirch and N.~Schappacher, \emph{Algebraic number theory}, Grundlehren der
  mathematischen Wissenschaften, Springer, Berlin, New York, Barcelona, 1999.

\bibitem{Odlyzko90}
A.~M. Odlyzko, \emph{Bounds for discriminants and related estimates for class
  numbers, regulators and zeros of zeta functions : a survey of recent
  results}, Journal de th{\'e}orie des nombres de Bordeaux \textbf{2} (1990),
  no.~1, 119--141 (eng).

\bibitem{RS62}
B.~Rosser and L.~Schoenfeld, \emph{Approximate formulas for some functions of
  prime numbers}, Illinois Journal of Mathematics \textbf{6} (1962), no.~1,
  64--94.

\bibitem{Schoof98}
R.~Schoof, \emph{Minus class groups of the fields of the {$\ell$}-th roots of
  unity}, Mathematics of Computation of the American Mathematical Society
  \textbf{67} (1998), no.~223, 1225--1245.

\bibitem{SV10}
N.~P. Smart and F.~Vercauteren, \emph{Fully homomorphic encryption with
  relatively small key and ciphertext sizes}, Public Key Cryptography, 2010,
  pp.~420--443.

\bibitem{Smith09}
B.~Smith, \emph{{Isogenies and the Discrete Logarithm Problem in Jacobians of
  Genus 3 Hyperelliptic Curves}}, {Journal of Cryptology} \textbf{22} (2009),
  no.~4, 505--529.

\end{thebibliography}

\end{document}